\documentclass[11pt,twoside]{amsart}
\usepackage[latin1]{inputenc}
\usepackage[T1]{fontenc}
\usepackage{mathtools}
\usepackage{graphicx}
\usepackage{tikz}
\usetikzlibrary{chains}
\usepackage{tcolorbox}

\usepackage[latin1]{inputenc}
\usepackage{amsmath}
\usepackage{amsthm}
\usepackage{amssymb}

\usepackage[all]{xy}
\usepackage{hyperref}
\newtheorem{thm}{Theorem}[section]

\newtheorem{prop}[thm]{Proposition}

\newtheorem{lem}[thm]{Lemma}
\newtheorem{cor}[thm]{Corollary}

\numberwithin{equation}{section}

\theoremstyle{definition}
\newtheorem{definition}[thm]{Definition}
\newtheorem{remark}[thm]{Remark}
\newtheorem{ex}[thm]{Example}
\usepackage[OT2,T1]{fontenc}
\DeclareSymbolFont{cyrletters}{OT2}{wncyr}{m}{n}
\DeclareMathSymbol{\Sha}{\mathalpha}{cyrletters}{"58}

\newcommand{\Db}{{\rm D}^{\rm b}}

\newcommand{\NS}{{\rm NS}}
\newcommand{\Pic}{{\rm Pic}}

\newcommand{\End}{{\rm End}}

\newcommand{\cal}{\mathcal}

\newcommand{\kc}{{\cal C}}

\newcommand{\ko}{{\cal O}}
\newcommand{\kp}{{\cal P}}

\newcommand{\ks}{{\cal S}}
\newcommand{\kt}{{\cal T}}
\newcommand{\kx}{{\cal X}}

\newcommand{\ZZ}{\mathbb{Z}}
\newcommand{\QQ}{\mathbb{Q}}
\newcommand{\RR}{\mathbb{R}}
\newcommand{\CC}{\mathbb{C}}

\newcommand{\HH}{\mathbb{H}}

\newcommand{\PP}{\mathbb{P}}

\renewcommand{\to}{\xymatrix@1@=15pt{\ar[r]&}}
\renewcommand{\rightarrow}{\xymatrix@1@=15pt{\ar[r]&}}
\renewcommand{\leftarrow}{\xymatrix@1@=15pt{&\ar[l]}}
\renewcommand{\mapsto}{\xymatrix@1@=15pt{\ar@{|->}[r]&}}
\renewcommand{\twoheadrightarrow}{\xymatrix@1@=18pt{\ar@{->>}[r]&}}
\renewcommand{\hookrightarrow}{\xymatrix@1@=15pt{\ar@{^(->}[r]&}}
\newcommand{\hook}{\xymatrix@1@=15pt{\ar@{^(->}[r]&}}
\newcommand{\congpf}{\xymatrix@L=0.6ex@1@=15pt{\ar[r]^-\sim&}}
\renewcommand{\cong}{\simeq}

\usepackage{comment}
\excludecomment{comment}
\begin{document}

\title[]{Brilliant families of K3 surfaces: Twistor spaces, Brauer groups, and Noether--Lefschetz loci}
\author[D.\ Huybrechts]{D.\ Huybrechts}

\address{Mathematisches Institut and Hausdorff Center for Mathematics,
Universit{\"a}t Bonn, Endenicher Allee 60, 53115 Bonn, Germany}
\email{huybrech@math.uni-bonn.de}

\begin{abstract} \noindent
We describe the Hodge theory of brilliant families of K3 surfaces. 
Their characteristic feature is a close link between the Hodge structures
of any two fibres over points in the Noether--Lefschetz locus.
Twistor deformations, the analytic Tate--{\v{S}}afarevi{\v{c}} group, and certain
one-dimensional Shimura special cycles are covered by the theory. In this setting,
the Brauer group is viewed as the
Noether--Lefschetz locus of the Brauer family or as the specialization of the
Noether--Lefschetz loci in a family of approaching twistor spaces. 
Passing from one algebraic twistor fibre to another, which by construction is  a transcendental operation, is here viewed as first deforming along the  more algebraic Brauer family and then along a family of algebraic K3 surfaces.

 \vspace{-2mm}
\end{abstract}

\maketitle
{\let\thefootnote\relax\footnotetext{The author is supported by the 
ERC Synergy Grant  HyperK (Grant agreement No.\ 854361).}
\marginpar{}
}
This note is concerned with the Hodge theory of brilliant families of K3 surfaces. A connected smooth holomorphic family $\ks\to C$ of K3 surfaces is called a
\emph{brilliant deformation}
of the fibre $S=\ks_0$ if there exists a class $\ell\in H^{1,1}(S,\ZZ)$ such that
under parallel transport one has
$$H^{2,0}(\ks_t)\subset H^{2,0}(S)\oplus H^{0,2}(S)\oplus\CC \cdot \ell$$ with $\ell$ not contained in $H^{2,0}(\ks_t)\oplus H^{0,2}(\ks_t)$, see \S \ref{sec:brilliant}  and \S \ref{sec:K3Transl} for details.

The starting point is 
the Hodge theory of  twistor spaces studied in \cite{HuyCM}, but 
more algebraic constructions like Brauer and Tate--{\v{S}}afarevi{\v{c}} groups as
well as special cycles of K3 surfaces are  covered by the theory. The Hodge theoretic
approach adopted here unifies these rather different geometric notions
and explains their interactions.
We start out by describing three particular  examples of brilliant families
and then state the main results in the geometric context.

\subsection{}  Examples of brilliant families include the following well-known
constructions:
\smallskip

\noindent
(i) The \emph{Dwork pencil} $\ks\to{\mathbb A}^1\setminus\{t\mid t^4=1\}$ of
quartic surfaces defined by $\sum_{i=0}^3x_i^4-4t\prod_{i=0}^3 x_i$,
the special fibre $S=\ks_0$ of which is the Fermat quartic.
\smallskip

\noindent
(ii) The \emph{Brauer} (or \emph{Tate--{\v{S}}afarevi{\v{c}}}) \emph{family} $\ks\to \CC\cong H^{0,2}(S)$ parameterizing all complex elliptic K3 surfaces
with a fixed elliptic K3 surface $S\to \PP^1$ as relative Jacobian.

\smallskip

\noindent
(iii) The \emph{twistor family} $\ks\to\PP^1$ associated with a polarized K3
surface $(S,L)$.
\smallskip

A priori, these examples have not much in common. For example, the Dwork pencil is an algebraic family, while for the Brauer and the twistor family most of the fibres are not
 projective.  However, as we shall explain, the Hodge theory describing the three families is similar in that  $H^{2,0}(\ks_t)$ varies in $H^{2,0}(S)\oplus H^{0,2}(S)\oplus\CC \cdot \ell$ for a certain algebraic class $\ell\in H^{1,1}(S,\ZZ)$. 
The three examples above correspond
to $(\ell.\ell)<0$ (Dwork), $(\ell.\ell)=0$ (Brauer), and $(\ell.\ell)>0$ (twistor). 

\begin{ex}\label{ex:3FamEll}
To an elliptic K3 surface $S\to \PP^1$ with a section
one can associate brilliant families of all three types, choosing
$\ell$ as the fibre class of the elliptic fibration, as an
ample class, or as any class of negative square, e.g.\ the class
of a component of a reducible fibre.
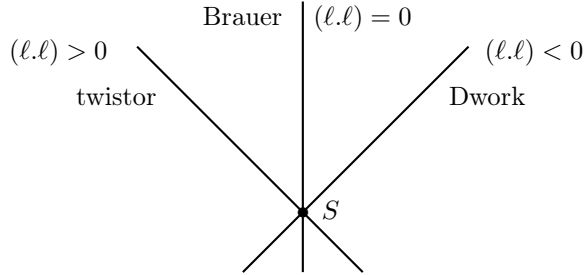
\begin{picture}(100,140)
\put(185,19){\begin{tikzpicture}  \draw[black, thick] (-1,-1) -- (2,2);\end{tikzpicture}}
\put(145,19){\begin{tikzpicture}  \draw[black, thick] (1,-1) -- (-2,2);\end{tikzpicture}}
\put(207.7,19){\begin{tikzpicture}  \draw[black, thick] (0,-1) -- (0,2.6);\end{tikzpicture}}
\put(97,100){\small $(\ell.\ell)>0$}
\put(212,113){\small $(\ell.\ell)=0$}
\put(277,100){\small $(\ell.\ell)<0$}
\put(205.6,39.5){\small $\bullet$}
\put(215.3,39.5){\small $S$}
\put(122.3,82.5){\small twistor}
\put(263.3,82.5){\small Dwork}
\put(170.3,113){\small Brauer}
 \end{picture}

A similar picture can be drawn for non-elliptic K3 surfaces by working with 
higher-dimensional moduli spaces of sheaves on K3 surfaces, see \S \ref{sec:higherDimMS}.
\end{ex}

\subsection{}
The first result is the observation that the main result in \cite{HuyCM}
generalizes from twistor spaces associated to K3 surfaces with complex multiplication
to general brilliant deformations of K3 surfaces with complex multiplication.

\begin{thm}\label{thm:mainRM}
Assume $\ks\to C$ is a brilliant deformation of a K3 surface $S$ with complex multiplication. Then any algebraic fibre $\ks_t$ with $\rho(\ks_t)=\rho(\ks)$
has complex multiplication as well and the maximal totally real subfields
of the Hodge endomorphism rings of the rational
transcendental lattices $T(S)\otimes\QQ$ 
and $T(\ks_t)\otimes\QQ$ are isomorphic.
\end{thm}

Beware that except for scalar multiplications, Hodge endomorphisms contained
in the maximal totally real subfield do not deform sideways in a brilliant family, i.e.\ they
are typically not realized as Hodge endomorphisms of the transcendental lattice of a fibre with $ \rho(\ks_t)<\rho(S)$.
\smallskip

We think of the countable set of all algebraic fibres $\ks_t$, $t\in C$, with
$\rho(\ks_t)=\rho(S)$ as the Noether--Lefschetz locus ${\rm NL}(\ks/C)\subset C$. Then the result can be rephrased by saying
that the maximal totally real subfield of $\End_{\rm Hdg}(T(S)\otimes \QQ)$ is reproduced as a field of Hodge endomorphisms
of the transcendental lattice of the fibre over any point in the Noether--Lefschetz locus.

 Viewing endomorphisms
of $T(S)\otimes\QQ$ as Hodge classes on $S\times S$, the assertion becomes more
geometric. However, for $(\ell.\ell)\ne0$ we are currently lacking a geometric explanation for the reappearance of these algebraic classes over all points
in the Noether--Lefschetz locus,
see \S \ref{sec:GeomBrauer}.

\subsection{} The Hodge theoretic characterization of brilliant families
reveals that in some appropriate sense the Noether--Lefschetz locus ${\rm NL}(\ks/\PP^1)$ of the
twistor family associated with $(S,L)$ specializes to the Brauer group ${\rm Br}(S)$.
This affirms and refines the point of view, suggested by F.\ Charles and further advocated by D.\ Bragg and M.\ Lieblich \cite{BL}, that the Brauer group of a supersingular K3 surface should be
viewed as the analogue of the twistor base in positive characteristic. We 
can use Hodge theory to make
this more precise in characteristic zero as follows. 

\begin{thm}\label{thm:NLvsBr}
Consider an elliptic K3 surface $S$ with polarizations $L_s$
approaching the fibre class of an elliptic fibration of $S$.
Then the twistor families $\ks(s)\to\PP^1$ associated with $(S,L_s)$ restricted
to the upper hemisphere of $S^2\cong\PP^1$ specialize to the Brauer family $\ks\to \CC$. Furthermore, the Noether--Lefschetz loci
${\rm NL}(\ks(s)/\PP^1)$ specialize to ${\rm Br}(S)$ up to the free abelian group
$H^2(S,\ZZ)/\NS(S)$.
\end{thm}

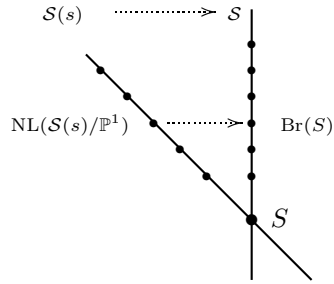
 \begin{picture}(100,120)
\put(145,19){\begin{tikzpicture}  \draw[black, thick] (1,-1) -- (-2,2);\end{tikzpicture}}
\put(207.7,19){\begin{tikzpicture}  \draw[black, thick] (0,-1) -- (0,2.6);\end{tikzpicture}}
\put(129,118){\tiny$\ks(s)$}
\put(199,118){\tiny$\ks$}
\put(149,118){$\xymatrix@C=15pt{\ar@{..>}[rr]&&}$}
\put(169,76){$\xymatrix@C=10pt{\ar@{..>}[rr]&&}$}
\put(219,76){\tiny${\rm Br}(S)$}
\put(117,76){\tiny${\rm NL}(\ks(s)/\PP^1)$}
\put(206,107){\tiny$\bullet$}
\put(206,97){\tiny$\bullet$}
\put(206,87){\tiny$\bullet$}
\put(206,77){\tiny$\bullet$}
\put(206,67){\tiny$\bullet$}
\put(206,57){\tiny$\bullet$}
\put(205.3,39.5){$\bullet$}
\put(149,97){\tiny$\bullet$}
\put(159,87){\tiny$\bullet$}
\put(169,77){\tiny$\bullet$}
\put(179,67){\tiny$\bullet$}
\put(189,57){\tiny$\bullet$}
\put(215.3,39.5){\small $S$}
 \end{picture}

The precise meaning of the notions involved in this statement will be explained in \S \ref{sec:TwBrHs} and \S \ref{sec:K3Transl}. Viewing the Brauer (or Tate--{\v{S}}afarevi{\v{c}}) family as a degeneration of twistor lines is not new and has
been discussed already by E.\ Markman \cite[Rem.\ 4.6]{Mark}, see also work of
M.\ Verbitsky \cite{Verb}. Our discussion adds the comparison of
Noether--Lefschetz loci and Brauer groups to the picture. 
The proof can be adapted to show analogously
that the Dwork family with its Noether--Lefschetz locus specializes to the Brauer family
and the Brauer group.

\subsection{} 
The three types of brilliant families, twistor, Brauer, and Dwork, can be put 
in a two-dimensional family. Those parameterizing algebraic K3 surfaces form
a countable union of curves, each of which intersects the three brilliant deformation types
in their Noether--Lefschetz locus. This allows one to view the transcendental twistor
construction as a combination of the Brauer family and a family of algebraic K3 surfaces.

\begin{thm}\label{thm:comp}
Assume $S$ is an elliptic K3 surface with an ample line bundle $L$. For every point
$t\in {\rm NL}(\ks/\PP^1)$ in the Noether--Lefschetz locus of the associated twistor
space $\ks\to\PP^1$, there exists a Brauer class $\alpha\in {\rm Br}(S)$ such that
the associated K3 surface $S_\alpha$ and $\ks_t$ are naturally
fibres of a holomorphic family of algebraic K3 surfaces.
\end{thm}

The interest in this viewpoint stems from the fact that we do understand 
the propagation of algebraic classes in Theorem \ref{thm:mainRM} for brilliant families
of Brauer type, see \S \ref{sec:GeomBrauer}, and that the propagation
along families of algebraic K3 surfaces should be more accessible to geometric arguments than for the purely transcendental twistor families. 

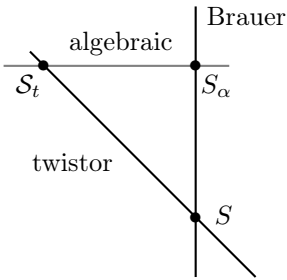
\begin{picture}(100,140)
\put(145,19){\begin{tikzpicture}  \draw[black, thick] (1,-1) -- (-2,2);\end{tikzpicture}}
\put(207.7,19){\begin{tikzpicture}  \draw[black, thick] (0,-1) -- (0,2.6);\end{tikzpicture}}
\put(135,99){\begin{tikzpicture}  \draw[gray, thick] (-1,2) -- (2,2);\end{tikzpicture}}
\put(148,97){\small $\bullet$}
\put(140,89){\small $\ks_t$}
\put(210,89){\small $S_\alpha$}
\put(205.5,97){\small $\bullet$}
\put(205.6,39.5){\small $\bullet$}
\put(215.3,39.5){\small $S$}
\put(146.3,59.5){\small twistor}
\put(160.3,105.5){\small algebraic}
\put(212.3,115){\small Brauer}
 \end{picture}

Once again, a similar result holds linking Dwork and Brauer families.

\smallskip

\noindent{\bf Acknowledgement:}  I would like to thank F.\ Charles for his interest and comments and the referee for valuable suggestions.

\section{Geometric families}\label{sec:Geom}
We start by reviewing some of the geometric properties of the three families mentioned 
in the introduction.

\subsection{}\label{sec:Dwork} The Dwork (or Fermat) pencil 
of quartic surfaces $\ks_t\subset \PP^3$ defined by $\sum_{i=0}^3x_i^4-4t\prod_{i=0}^3 x_i$ parameterizes K3 surfaces for $t^4\ne 1$. The four singular surfaces $\ks_t$, $t\in\{\pm1,\pm i\}$, have 16 ordinary double points and their minimal resolutions are
K3 surfaces. In fact,
by passing to a double cover $C\to {\mathbb A}^1$ ramified at the four points
and then taking a small resolution of the singular base change, the Dwork pencil can be turned into an algebraic family of quasi-polarized K3 surfaces.

The Fermat quartic $S=\ks_0$ is known to have maximal Picard number $\rho(S)=20$,
in fact its N\'eron--Severi group $\NS(S)\cong\ZZ^{\oplus 20}$ is generated by the lines contained in $S$,
and its transcendental lattice $T(S)$ is of rank two with intersection matrix ${\rm diag}(8,8)$,
see \cite{SSL}.  Furthermore, there exists a primitive sublattice $\ell^\perp\subset\NS(S)$ of corank one that stays algebraic along the whole family and $\ell^\perp\cong \NS(\ks_t)$ for very general $t$. By Hodge index theorem, $(\ell.\ell)<0$ and, more precisely, $\ell$ can be chosen of the form $L_1+L_2-L_3-L_4$ for two disjoint pairs of intersecting lines $L_i\subset S$ such that $(\ell.\ell)=-4$.\footnote{Thanks to E.\ Sert\"oz for this information.} Thus, under parallel transport,
$$H^{2,0}(\ks_t)\subset H^{2,0}(S)\oplus H^{0,2}(S)\oplus \CC\cdot \ell=T(S)_\CC\oplus\CC\cdot\ell.$$

For another example, consider a non-trivial deformation $E_t$ of a CM elliptic curve $E$
 and let  $\ks_t$ be the family of Kummer surfaces associated with
$E_t\times E_t$. Again, $\rho(\ks_t)=19$ for the very general $E_t$ and $\rho(\ks_t)=20$ whenever $E_t$ has complex multiplication.
Clearly, $H^{2,0}(\ks_t)=H^{1,0}(E_t)\otimes H^{1,0}(E_t)\subset H^2(E_t\times E_t,\CC)
\subset H^2(\ks_t,\CC)$
is contained in the three-dimensional space orthogonal to the  classes
$[E_t\times\{{\rm pt}\}],[\{{\rm pt}\}\times E_t],[\Delta_t]\in H^2(E_t\times E_t,\ZZ)$
or, in other words, $$
H^{2,0}(\ks_t)\subset (H^{1,0}(E)\otimes H^{1,0}(E))\oplus(H^{0,1}(E)\otimes H^{0,1}(E))\oplus \CC\cdot \ell$$ where $\ell$ is an appropriate linear combination of the
 graph $\Gamma$ of a complex multiplication of $E$ and the aforementioned  three algebraic classes.
\smallskip

More generally, any family of projective K3 surfaces $\ks_t$ for which the very general fibre has Picard number $\rho(\ks_t)=19$ is of the type covered by our considerations.
These families yield one-dimensional Shimura special cycles in the moduli space of polarized K3 surfaces.
It is possible to construct holomorphic families  of projective K3 surfaces $\ks_t$
over a one-dimensional base that also fit our theory and have $\rho(\ks_t)<19$ for the very general fibre $\ks_t$, but then the family does not correspond
to an algebraic curve in the moduli space, cf.\ Remark \ref{rem:AO}.

\subsection{}\label{sec:GeomBrauer} We next come to the Brauer family.
Consider a K3 surface $S$ together with a fibration
$\pi\colon S\to\PP^1$ by curves of genus one. We shall assume that
the fibration has a section, so that all smooth fibres are elliptic curves,
in which case $\pi$ is called an elliptic fibration. The Fermat quartic is
an example of an elliptic K3 surface and it is so in more than one way.

The Tate--{\v{S}}afarevi{\v{c}} group of $S$ parameterizes all K3 surfaces with a
genus one fibration and an isomorphism of its relative Jacobian fibration with $S$. It
comes in two flavors: The algebraic and the analytic Tate--{\v{S}}afarevi{\v{c}} group. More precisely, $\Sha^{\rm an}(S)$ is the set of all isomorphism classes of pairs
 $(S',\psi)$ consisting of a complex K3 surface  $S'$, typically not projective,
 with a fixed fibration $S'\to \PP^1$ by curves of genus one and
   an isomorphism $\psi\colon J(S'/\PP^1)\congpf S$ over $\PP^1$
 between the relative Jacobian $J(S'/\PP^1)$ and $S$, see \cite[Ch.\ 1.5]{FM} or \cite[Ch.\ 11.5]{HuyK3}. The algebraic version is $\Sha(S)$, which is defined similarly but now
 requiring $S'$ to be projective. Clearly, then
 $$\Sha(S)\subset \Sha^{\rm an}(S).$$
Both sets are indeed groups and as such they are naturally
isomorphic to the algebraic resp.\ analytic Brauer groups
$$\xymatrix@C=6pt@R=0pt{\Sha(S)&\subset& \Sha^{\rm an}(S)\\
|\wr&&|\wr\\
{\rm Br}(S)&\subset&{\rm Br}^{\rm an}(S).}$$
Here, the algebraic Brauer group ${\rm Br}(S)$ can  be described cohomologically
as the torsion subgroup of $H^2(S,\ko_S^\ast)$ or, more directly, as $H^2_{\text{\'et}}(S,{\mathbb G}_m)$, while the analytic Brauer group is all of $H^2(S,\ko_S^\ast)$.
From the exponential sequence one obtains an exact sequence
$$\xymatrix@C=15pt{0\ar[r]&\Pic(S)\ar[r]&H^2(S,\ZZ)\ar[r]&H^2(S,\ko_S)\ar[r]&H^2(S,\ko_S^\ast)\ar[r]&0,}$$
which leads to the description of the analytic Brauer group
as $${\rm Br}^{\rm an}(S)\cong H^2(S,\ko_S)/{\rm coker}\!\left(\NS(S)\,\hookrightarrow H^2(S,\ZZ)\right)\cong \CC/\ZZ^{\oplus 22-\rho(S)}.$$
Note that the obvious inclusion $T(S)\subset H^2(S,\ZZ)/\NS(S)$ is of finite index
 and, therefore, the natural surjection
$$H^{0,2}(S)/T(S)\cong H^2(S,\ko_S)/T(S)\twoheadrightarrow {\rm Br}^{\rm an}(S)$$ has a finite, typically non-trivial,  kernel.
From a Hodge theoretic perspective $H^2(S,\ko_S)/T(S)$ is more natural, but the difference between the two  groups will be of no importance to us. 
Note that unless $\rho=20$, the groups $$\Sha^{\rm an}(S)\cong {\rm Br}^{\rm an}(S)\cong H^{0,2}(S)/\ZZ^{\oplus 22-\rho(S)}$$ have no reasonable geometric structure and cannot
serve as a basis for a family of  all K3 surfaces $S'$ parameterized by $\Sha^{\rm an}(S)$.
However, there exists a family, the \emph{Brauer family}
\begin{equation}\label{eqn:BrauerFam}
\ks\to H^{0,2}(S)\cong \CC,
\end{equation}
for which the fibre $\ks_t$ over $t\in H^{0,2}(S)$ is isomorphic to the K3 surface
$S'$ corresponding to the image $(S',\psi)$ of $t$ under $H^{0,2}(S)\twoheadrightarrow {\rm Br}(S)\cong\Sha^{\rm an}(S)$, see \cite[Ch.\ 1.5]{FM} for a detailed discussion. Note
that $\ks\to \CC$ is actually a family of K3 surfaces together with a genus one fibration
which is the one that is given by considering $(S',\psi)$ as an element
in $\Sha^{\rm an}(S)$.

By construction, we again have
$$H^{2,0}(\ks_t)\subset H^{2,0}(S)\oplus\CC\cdot  \ell\subset H^{2,0}(S)\oplus H^{0,2}(S)\oplus\CC\cdot \ell,$$
where  $\ell=f\in H^2(S,\ZZ)$ denotes the class of a fibre of the elliptic
fibration $S\to \PP^1$. The construction does not provide a topological or $\kc^\infty$-trivialization, but see Remark \ref{rem:TrivBrauer}.

\begin{remark}
Instead of studying the Brauer family $\ks\to \CC\cong H^{0,2}(S)$ one can
consider the constant family $S\times\CC\to \CC$ and endow it with
a universal Brauer class ${\pmb\alpha}\in H^2(S\times\CC,\ko^\ast)$. Explicitly,
using $H^2(S\times \CC,\ko)\cong H^0(\CC,\ko)\otimes H^{0,2}(S)$,
which under the identification $\CC=H^{0,2}(S)$ contains $H^{0,2}(S)^\ast\otimes H^{0,2}(S)$, 
the universal Brauer class ${\pmb\alpha}$ is realized as the image of
${\rm id}_{H^{0,2}}$ under the exponential map $H^2(S\times \CC,\ko)\to H^2(S\times\CC,\ko^\ast)$.
By construction, if $\alpha\in H^2(S,\ko^\ast_S)$ is the image of $t\in \CC=H^{0,2}(S)$,
then ${\pmb \alpha}|_{S\times\{t\}}=\alpha$. The situation was discussed for supersingular K3 surfaces in positive characteristic in \cite{BL}.
\end{remark}

\begin{remark}
The analogy between the Brauer groups of supersingular K3 surfaces
and twistor spaces promoted in \cite{BL} is sometimes met with the following objection.
The Hochschild cohomology $H\!H^2(S)\cong H^2(S,\ko_S)\oplus H^1(S,\kt_S)\oplus H^0(S,\bigwedge^2\kt_S)$ parameterizes all infinitesimal deformations of $S$, classical commutative as
well as non-commutative ones. First order classical deformations correspond
to classes in $H^1(S,\kt_S)$, while non-commutative ones associated to
deformations of the standard bounded derived category $\Db(S)$ of coherent sheaves to
the bounded derived category $\Db(S,\alpha)$ of $\alpha$-twisted coherent sheaves
correspond to classes in $H^2(S,\ko_S)$. How could then possibly 
the tangent space of ${\rm Br}^{\rm an}(S)$, which is naturally identified
with $H^2(S,\ko_S)$, end up in the direct summand $H^1(S,\kt_S)$?

The answer to this is that it is actually not $S$ that is deformed in the Brauer family above
but its Fourier--Mukai partner, which just happens to be isomorphic to $S$. More precisely, the automorphism of  $H\!H^2(S)$ induced by the relative Poincar\'e sheaf $\kp$ on $S\times_{\PP^1}S$  is compatible with the action of $\kp$
on Hochschild homology $H\!H_\ast(S)\cong \bigoplus_{q-p=\ast} H^{p,q}(S)$
considered as a module over $H\!H^\ast(S)$.
Identifying $H\!H_\ast(S)$ with (ungraded) de Rham cohomology $H^\ast(S,\CC)$
and using that the action of $\kp$ on $H^\ast(S,\ZZ)$ identifies $(H^0\oplus H^4)(S,\ZZ)$
with the hyperbolic plane in ${\rm NS}(S)\subset H^2(S,\ZZ)$ spanned by
the fibre class $\ell$ and the class of the section, one finds
that $\kp$ indeed sends $H^2(S,\ko_S)\subset H\!H^2(S)$ into $H^1(S,\kt_S)\subset H\!H^2(S)$. 

A similar point of view is exploited in \S \ref{sec:higherDimMS}, where non-commutative deformations of $S$ corresponding to Brauer classes are
interpreted in terms of classical commutative deformations of a certain moduli space
of sheaves on $S$.
\end{remark}

For $\alpha\in {\rm Br}(S)$, let $(S_\alpha,\psi_\alpha)\in \Sha(S)$ be the corresponding
algebraic K3 surface $\pi_\alpha\colon S_\alpha\to\PP^1$ together with an isomorphism $\psi_\alpha\colon J(S_\alpha/\PP^1)\congpf S$. The relative Jacobian $J(S_\alpha/\PP^1)$ can be interpreted as a moduli space of stable sheaves concentrated on the fibres of $\pi_\alpha$, but it is not a fine moduli space. In fact, the obstruction to the existence of a universal family is exactly
$\alpha\in {\rm Br}(S)$. In other words,  a universal sheaf $\kp_\alpha$ on 
$S_\alpha\times_{\PP^1}S$ exists but only as a twisted sheaf, where the twist
is with respect to $1\boxtimes\alpha\in {\rm Br}(S_\alpha\times S)$.  The usual Fourier--Mukai formalism then yields
an exact equivalence
$$\Db(S_\alpha)\cong \Db(S,\alpha)$$
between the bounded derived category $\Db(S_\alpha)$ of coherent sheaves on $S_\alpha$ and the bounded derived category $\Db(S,\alpha)$ of
$\alpha$-twisted coherent sheaves on $S$, for more information and references
see \cite[\S 10.2.2\,\&\,Rem.\ 11.5.9]{HuyK3}. The Mukai vector of the twisted universal
sheaf yields a Hodge isometry $\widetilde H(S_\alpha,\ZZ)\cong \widetilde H(S,\alpha,\ZZ)$,
which restricts to a Hodge isometry between their transcendental lattices $T(S_\alpha)\cong T(S,\alpha)$, see \cite{HuySeattle,HuyStellari} for details. Since the transcendental lattice of the twisted K3 surface $(S,\alpha)$
is a finite index sub-Hodge structure of $T(S)$, this yields  a Hodge isometry
\begin{equation}\label{eqnTSSalpha}
T(S_\alpha)_\QQ\cong T(S)_\QQ.
\end{equation}
Moreover, (\ref{eqnTSSalpha}) is algebraic, i.e.\ it is defined by an algebraic class
in $(T(S_\alpha)\otimes T(S))_\QQ^{2,2}\subset H^{2,2}(S_\alpha\times S,\QQ)$.
Here, we use that the $(1\boxtimes \alpha)$-twisted sheaf $\kp$ becomes
untwisted by passing to its derived tensor power $\kp^{\otimes r}$ for $r={\rm ord}(\alpha)$. The Chern classes of the untwisted sheaves $\kp^{\otimes r}$ exist in the
traditional sense, which eventually leads to a description of (\ref{eqnTSSalpha}) in terms
of algebraic classes associated to untwisted sheaves. 

\begin{remark}
There is a slightly confusing point about the relation between the transcendental lattices
$T(S_\alpha)$, $T(S)$, and $T(S,\alpha)$. By definition,
$T(S_\alpha)_\QQ\subset H^2(S_\alpha,\QQ)$ and $T(S)_\QQ\subset H^2(S,\QQ)$, while $T(S,\alpha)_\QQ\subset \widetilde H(S,\alpha,\QQ)$ is more naturally viewed as
$T(S,\alpha)_\QQ=\exp(B)\cdot T(S)_\QQ\subset (H^2\oplus H^4)(S,\QQ)$, where $B\in H^2(S,\QQ)$ lifts $\alpha$. The cohomological Fourier--Mukai formalism above then yields
$T(S_\alpha)_\QQ\cong T(S,\alpha)_\QQ\subset (H^2\oplus H^4)(S,\QQ)$. On the other hand,
parallel transport along the Brauer family leads to an inclusion $T(S_\alpha)_\QQ\subset T(S)_\QQ\oplus \QQ\cdot f\subset H^2(S,\QQ)$, where $f$ is the class of the fibre of the elliptic fibration $S\to \PP^1$. Furthermore, the image  of $T(S_\alpha)\cong T(S,\alpha)\subset (H^2\oplus H^4)(S,\ZZ)$ under the projection onto $H^2(S,\ZZ)$ identifies $T(S_\alpha)$ with the kernel of $T(S)\to \QQ/\ZZ$, $\gamma\mapsto (B.\gamma)$, which eventually leads to
the Hodge isometry $T(S_\alpha)_\QQ\cong T(S)_\QQ$.
In particular, the one-dimensional $H^{2,0}(S_\alpha)\subset H^2(S_\alpha,\CC)$ is via
(\ref{eqnTSSalpha}) mapped into $H^{2,0}(S)\oplus H^4(S,\CC)$ while
parallel transport maps it into $H^{2,0}(S)\oplus\CC\cdot f$.
\end{remark}

\begin{remark} The preceding discussion can be made to work with suitable
modifications in the case when $\alpha$ is not torsion. Again $S$ can be viewed as
a coarse moduli space of torsion sheaves on $S_\alpha$ for which the universal
sheaf is twisted with respect to the non-torsion class $\alpha$ on $S$.
For moduli spaces of stable sheaves on non-projective K3 surfaces see \cite{PeregoToma}.\footnote{Note however that  the case of  torsion sheaves needs extra care, but as it is not essential for our purpose, we ignore this point.}
There is still a Hodge isometry $T(S_\alpha)\cong T(S,\alpha)$, only that now $T(S,\alpha)$
cannot be viewed just as the $B$-field shift of $T(S)$ (up to finite index). By definition, it
is the minimal primitive sub-Hodge structure of $H^2(S,\ZZ)\oplus H^4(S,\ZZ)$ with its $(2,0)$-part spanned by $\sigma+\sigma\wedge t$, where $t\in \CC\cong H^{0,2}(S)$ lifts $\alpha$ under $H^{0,2}(S)\to {\rm Br}^{\rm an}(S)$. In fact, when
$\alpha$ is non-torsion, the rank of the transcendental lattice $T(S_\alpha)$ exceeds the
rank  of $T(S)$ and, moreover, the intersection form is degenerate on it.
\end{remark}

Clearly, since $S_\alpha$ and $S$ have Hodge isometric rational transcendental lattices
for $\alpha\in{\rm Br}(S)$,
their fields of Hodge endomorphisms are isomorphic:
$$\End_{\rm Hdg}(T(S_\alpha)\otimes \QQ)\cong \End_{\rm Hdg}(T(S)\otimes\QQ).$$
Furthermore, viewing $S$ as a moduli space of sheaves on $S_\alpha$ or, conversely,
$S_\alpha$ as a moduli space of $\alpha$-twisted sheaves on $S$, provides us
with a geometric understanding of this isomorphism.
The Mukai vector of (some tensor power of) $\kp\boxtimes\kp$  on $(S_\alpha\times S)^2$ maps
any Hodge endomorphism $\varphi\in \End_{\rm Hdg}(T(S)\otimes\QQ)$ or, equivalently, any Hodge class $\varphi \in(T(S)\otimes T(S))^{2,2}_\QQ\subset H^{2,2}(S\times S,\QQ)$  to a Hodge class on $S_\alpha\times S_\alpha$. Under this map, algebraic classes are mapped to algebraic
classes. Note, however, that the Hodge conjecture for squares of K3 surfaces is only known for K3 surfaces with CM endomorphism field.

We summarize the above discussion as follows.

\begin{prop}
Consider the Brauer family (\ref{eqn:BrauerFam}) $\ks\to\CC$ associated
with an elliptic K3 surface $S\to \PP^1$. Then for any algebraic fibre $\ks_t$ there
exists an algebraic Hodge isometry
$$T(S)\otimes \QQ\cong T(\ks_t)\otimes\QQ.$$
Furthermore, the induced isomorphism of their endomorphism fields
$$\End_{\rm Hdg}(T(S)\otimes \QQ)\cong\End_{\rm Hdg}(T(\ks_t)\otimes\QQ)$$
maps algebraic classes to algebraic classes.\qed
\end{prop}

In \S \ref{sec:higherDimMS}
we explain that the same principle applies to projective K3 surfaces $S$
which are not elliptic by working with higher-dimensional moduli spaces. In the case of elliptic K3 surfaces one can state more algebraically that the rational Chow motives of
all algebraic fibres of the Brauer family are isomorphic: ${\mathfrak h}(\ks_t)\cong {\mathfrak h}(S)$, cf.\ \cite{HuyHC}.

\begin{remark}
(i) The first assertion in the above proposition does not hold for the other two types of families considered in this article. A weaker version of the second assertion does hold, see Proposition \ref{prop:CMBrilliant}, but a clear geometric reason for the recurrence of the same Hodge classes in all algebraic fibres is lacking.

(ii) Furthermore, for the Brauer family, we also know that if $S$ is defined
over $\bar\QQ$, also all other algebraic fibres $S_\alpha$ are defined over $\bar\QQ$.
This is unknown for the algebraic fibres of the twistor family associated to a K3 surface without complex multiplication.
\end{remark}

\subsection{}\label{ref:Twist} We conclude by a short review of twistor families which have been discussed already in the prequel \cite{HuyCM}.
To any K3 surface $S$ together with a K\"ahler class $\omega\in H^{1,1}(S,\RR)$ there is naturally associated a complex structure on $S\times\PP^1$, with the corresponding complex threefold  denoted $\ks$, such that the projection defines
a holomorphic map $\ks\to\PP^1$. The fibres $\ks_t$ are K3 surfaces with
$$H^{2,0}(\ks_t)\subset H^{2,0}(S)\oplus H^{0,2}(S)\oplus \CC\cdot\omega\subset T(S)_\CC\oplus\CC\cdot\omega.$$
The construction of $\ks\to \PP^1$, the twistor family associated with $(S,\omega)$, 
depends on the existence of the unique Ricci flat K\"ahler structure with K\"ahler class
$\omega$. In this sense, it is a transcendental construction. For more information and references see  \cite[\S 1]{HuyCM}. 

For a projective K3 surface $S$, the construction
can be applied to any ample line bundle $L$ on $S$ whose first Chern class $\ell\coloneqq {\rm c}_1(L)$ is of course a K\"ahler class. However, even in this case,
only countably many of the fibres $\ks_t$ are projective. According to \cite[Prop. 3.2]{HuyCM}, any projective fibre $\ks_t$ for which $H^{2,0}(\ks_t)\oplus H^{0,2}(\ks_t)$ does not contain $\ell$ satisfies $\rho(\ks_t)=\rho(S)$. The set of all fibres for which
$\ell$ is contained in $H^{2,0}(\ks_t)\oplus H^{0,2}(\ks_t)$ can be pictured as the equator
of the sphere $\PP^1\cong S^2$ with the original K3 surface $S$ corresponding to the north pole.

\begin{remark}\label{rem:TrivBrauer}
Inspired by the twistor construction one might try to find a $\kc^\infty$-trivialization
of the Brauer family. This is indeed possible but  not in a canonical way. Concretely, if
$\pi\colon S\to\PP^1$ is an elliptic K3 surface and $S$ is viewed as a differentiable manifold
$M$ endowed with a complex structure $I$, then a generator $\sigma$ of $H^{2,0}(S)$ is
a closed complex two-form on $M$ which uniquely determines the complex structure
$I$, for $T^{0,1}(M,I)\subset T_\CC M$ is the kernel of $\sigma\colon T_\CC M\to T_\CC^\ast M$. Now, pick any non-exact closed $(1,1)$-form $\omega$ on $\PP^1$, say with $\int\omega=1$. Then $\sigma_t\coloneqq\sigma+t\pi^\ast\omega$ is a closed complex two-form with the same properties as $\sigma$, namely $\sigma_t\wedge\sigma_t\equiv 0$ and
$\sigma_t\wedge\bar\sigma_t>0$ (pointwise), and therefore defines a complex structure
$I_t$ on $M$. Note that with respect to $I_t$ the projection $\pi$ is still holomorphic and, hence, $S_t\coloneqq (M,I_t)$ comes with a natural genus one fibration. However, the
section of $\pi\colon S\to \PP^1$ will not be holomorphic anymore with respect to $I_t$. Altogether,
this yields a family of complex structures $I_t$ on $M$ that is isomorphic
to the Brauer family. We emphasize that the construction of the family $(M,I_t)$ 
depends on the choice of $\omega$ and in this sense the Brauer family does not come with a natural trivialization. Also note that proving that
the relative Jacobian of $S_t$ is isomorphic to $S$ for all $t$ is not immediate.
\end{remark}
\section{brilliant families of Hodge structures}\label{sec:brilliantHodge}
This section deals with all purely Hodge theoretic aspects. The results can be applied
to any geometric situation that involves Hodge structures of weight two with a one-dimensional $(2,0)$-part. Later we will focus exclusively on K3 surfaces, but the Hodge theory developed here equally well applies to hyperk\"ahler manifolds.

\subsection{}\label{sec:brilliant}
We start by considering a real vector space $T_\RR$ with a non-degenerate
symmetric bilinear form $(~.~)$ of signature $(2,r-2)$. A Hodge structure of K3 type on $T_\RR$ is then given by a generator $\sigma_0\in T_\RR\otimes\CC$ of its $(2,0)$-part, unique up to scaling, that satisfies the two conditions 
$(\sigma_0.\sigma_0)=0$ and $(\sigma_0.\bar\sigma_0)>0$. Alternatively, the Hodge structure
can be thought of in terms of the oriented, positive plane $P_{\sigma_0}\coloneqq\langle{\rm Re}(\sigma_0),{\rm Im}(\sigma_0)\rangle\subset T_\RR$.

Next, we extend $T_\RR$ to the orthogonal sum
\begin{equation}\label{eqn:TRplusell}
T_\ell\coloneqq T_\RR\oplus\RR\cdot\ell,
\end{equation} where $d\coloneqq(\ell.\ell)\in\RR$
can be arbitrary. The discussions in the three cases, $d>0$, $d=0$, and $d<0$, will be similar, but there are interesting differences and  special phenomena  that we wish to explore.

For now we fix a Hodge structure on $T_\RR$ with $\sigma_0\in T_\RR\otimes\CC$ as above and consider it as a Hodge structure on $T_\ell$ by declaring $\ell$ to be of type $(1,1)$. 
We are interested in Hodge structures of K3 type on $T_\ell$, i.e.\
for which the $(2,0)$-part is spanned by a class $\sigma_t\in T_\ell\otimes\CC$ with $(\sigma_t.\sigma_t)=0$ and $(\sigma_t.\bar\sigma_t)>0$. Once again, even for $d=0$,
such a Hodge structure is uniquely determined by $\sigma_t$, for its $(1,1)$-part
is the orthogonal complement of the positive plane $P_{\sigma_t}\coloneqq\langle
{\rm Re}(\sigma_t),{\rm Im}(\sigma_t)\rangle\subset T_\RR\oplus \RR\cdot \ell$.

However, we shall  only consider Hodge structures for which 
$$\sigma_t\in \CC\cdot\sigma_0\oplus\CC\cdot\bar \sigma_0\oplus\CC\cdot\ell$$
or, equivalently, such that $P_{\sigma_t}\subset P_{\sigma_0}\oplus\RR\cdot\ell$. They
are parameterized by an open subset $D_\ell\subset Q_\ell$ of the conic $
Q_\ell\subset\PP(\CC\cdot\sigma_0\oplus\CC\cdot\bar \sigma_0\oplus\CC\cdot\ell)$
defined by $(\sigma.\sigma)=0$:
\begin{equation}\label{eqn:Dell}
D_\ell\coloneqq\{\,\sigma\in Q_\ell\mid (\sigma.\bar\sigma)>0\,\}\subset Q_\ell\subset \PP(\CC\cdot\sigma_0\oplus\CC\cdot\bar \sigma_0\oplus\CC\cdot\ell)\cong\PP^2.
\end{equation}
For $d\ne 0$ the conic $Q_\ell$ is smooth, while for $d=0$ it consists of the two conjugate lines $\PP(\CC\cdot\sigma_0\oplus\CC\cdot\ell)$ and $
\PP(\CC\cdot\bar\sigma_0\oplus\CC\cdot\ell)$
with $[\ell]$ as their point of intersection. In the latter case, we shall write
$L_\ell\coloneqq\PP(\CC\cdot\sigma_0\oplus\CC\cdot\ell)\setminus \{[\ell]\}$.
We emphasize that, although not reflected by the notation, $D_\ell$ and $Q_\ell$ depend
not only on $T_\RR$ and $\ell$, but also on $\sigma_0$.

\begin{lem}\label{lem:PeriodDomain}
Let $T_\RR\subset T_\RR\oplus\RR\cdot \ell$ be as above. Then
$$D_\ell=~Ê\begin{cases}~Q_\ell\cong \PP^1&\text{ if } d>0\\[5pt]
~L_\ell\sqcup \bar L_\ell\cong \CC\sqcup \CC&\text{ if }Ê d=0\\[5pt]
~ D_\ell'\sqcup\bar D_\ell'Ê\cong \HH\sqcup \bar\HH&\text{ if }~ d<0,
 \end{cases}$$
where the connected component $D_\ell'$ is chosen to contain $\sigma_0$.
\end{lem}
\begin{proof} 
The cases $d>0$ and $d=0$ are covered by the above discussion. For $d<0$ observe that
$\langle{\rm Re}(\sigma_0),{\rm Im}(\sigma_0),\ell\rangle_\RR$ is isometric to $\RR^3$
with $e_1,e_2$ isotropic, $(e_1.e_2)=1$, and $(e_3.e_3)=1$. Then  $z\mapsto [1:-(z.z):\sqrt 2 z]$ describes a biholomorphic identification of $\HH$ with one of the two connected components of the period domain $D_\ell$.
\end{proof}

To restore the symmetry between the three cases, we shall shrink
the period domain  for $d>0$ to the
complement of the equator $S^1_\ell\subset D_\ell\cong\PP^1$. Recall that the equator
is described by the condition $\ell\in P_{\sigma_t}$. Hence,
$$D_\ell\setminus S^1_\ell\cong \HH\sqcup \bar\HH$$ as in the case $d<0$ and we
usually follow the convention that $\sigma_0\in \HH$. 

\begin{remark}\label{rem:equatorflow}
Note that for $d\leq 0$ every  $\sigma_t\in D_\ell$ satisfies $\ell\not\in P_{\sigma_t}$. 
In fact, if one lets
$d>0$ approach $d=0$, then the equator $S_\ell^1$ flows into the point of
intersection $[\ell]$. Thus, removing the equator allows one to view the two connected
components in each of the three cases as part of a family depending on the parameter $d=(\ell.\ell)$.

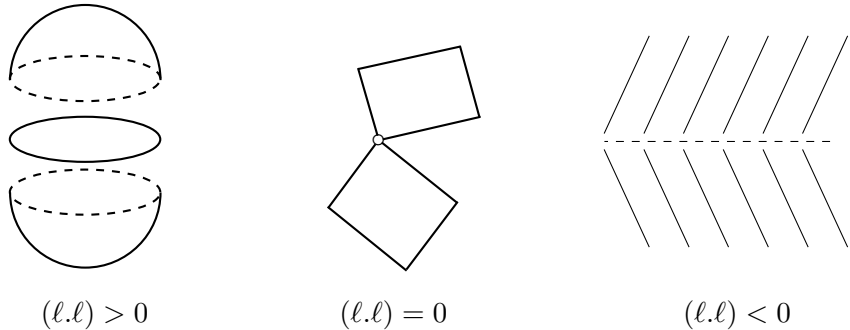
\begin{picture}(100,140)
\put(55,93){\tikz \draw[black,thick,dashed] (0,0) ellipse (1cm and 0.3cm);}
\put(55,70){\tikz \draw[black,thick] (0,0) ellipse (1cm and 0.3cm);}
\put(55,50){\tikz \draw[black,thick,dashed] (0,0) ellipse (1cm and 0.3cm);}
\put(55,101){\tikz\draw[black,thick] (0,0) arc (0:180:1cm);}
\put(55,30){\tikz\draw[black,thick] (0,0) arc (0:-180:1cm);}

\put(175.2,28.8){\tikz \draw[black, thick] (0.55,-0.55) -- (-0.5,0.3);}
\put(195.7,54.6){\tikz \draw[black, thick] (0.5,-0.5) -- (-0.5,0.3);}
\put(175.3,52.6){\tikz \draw[black, thick] (0.15,-0.55) -- (0.78,0.3);}
\put(205,29.4){\tikz \draw[black, thick] (0.1,-0.6) -- (0.78,0.3);}
\put(196,78.6){\tikz \draw[black, thick] (0.8,0.6) -- (-0.5,0.3);}
\put(186.7,105.2){\tikz \draw[black, thick] (0.86,0.6) -- (-0.5,0.3);}
\put(186.6,79.9){\tikz \draw[black, thick] (0.1,-0.6) -- (-0.16,0.3);}
\put(225.2,87){\tikz \draw[black, thick] (0.1,-0.65) -- (-0.16,0.3);}

\put(280,78){\tikz \draw [dashed] (3,1) -- (0,1);}
\put(280,38){\tikz \draw[black, thin] (0.3,-0.3) -- (-0.3,1);}
\put(295,38){\tikz \draw[black, thin] (0.3,-0.3) -- (-0.3,1);}
\put(310,38){\tikz \draw[black, thin] (0.3,-0.3) -- (-0.3,1);}
\put(325,38){\tikz \draw[black, thin] (0.3,-0.3) -- (-0.3,1);}
\put(340,38){\tikz \draw[black, thin] (0.3,-0.3) -- (-0.3,1);}
\put(355,38){\tikz \draw[black, thin] (0.3,-0.3) -- (-0.3,1);}

\put(280,81){\tikz \draw[black, thin] (-0.3,-0.3) -- (0.3,1);}
\put(295,81){\tikz \draw[black, thin] (-0.3,-0.3) -- (0.3,1);}
\put(310,81){\tikz \draw[black, thin] (-0.3,-0.3) -- (0.3,1);}
\put(325,81){\tikz \draw[black, thin] (-0.3,-0.3) -- (0.3,1);}
\put(340,81){\tikz \draw[black, thin] (-0.3,-0.3) -- (0.3,1);}
\put(355,81){\tikz \draw[black, thin] (-0.3,-0.3) -- (0.3,1);}
\put(67,10){$(\ell.\ell)>0$}
\put(180,10){$(\ell.\ell)=0$}
\put(192,76){$\circ$}
\put(310,10){$(\ell.\ell)<0$}

 \end{picture}
 
 Those are the deformations of $\sigma_0$ that interest us here.
\end{remark}

\begin{definition}\label{def:brilliantHS}
For a fixed Hodge structure of K3 type on $T_\RR$ given by $\sigma_0\in T_\RR\otimes\CC$ the family of Hodge structures $\sigma_t$ on $T_\RR\oplus\RR\cdot\ell$ satisfying $$P_{\sigma_t}\subset P_{\sigma_0}\oplus\RR\cdot \ell\text{ and }\ell\not\in P_{\sigma_t}$$ is called a \emph{brilliant family}
of deformations  of $\sigma_0$. We usually restrict to the connected component containing $\sigma_0$, the other one being obtained by complex conjugation.
\end{definition}

Occasionally, we will refer to the three types of brilliant deformations corresponding to
$d<0$, $d=0$, and $d>0$ as deformations of \emph{Dwork}, \emph{Brauer}, resp.\ \emph{twistor type}.

\subsection{}\label{sec:RatHS} In a second step, we start with a $\QQ$-vector space $T$, associate to it the real vector space
$T_\RR\coloneqq T\otimes\RR$, and assume that the Hodge structure $\sigma_0$
is irreducible. In particular, orthogonal projection yields
$\QQ$-linear injections $T\,\hookrightarrow P_{\sigma_0}$ and $T\,\hookrightarrow \CC\cdot \sigma_0$, cf.\ \cite[Lem.\ 2.3]{HuyCM}.
Furthermore, the class $\ell$ is now considered to be rational and of type $(1,1)$, so that $\sigma_0$ 
can be viewed as a Hodge structure on the $\QQ$-vector space 
\begin{equation}\label{eqn:TQplusell}
T\oplus\QQ\cdot\ell.
\end{equation} 
For a brilliant deformation $\sigma_t$ of $\sigma_0$ on $T\oplus\QQ\cdot\ell$
we denote by $$T_t\subset T\oplus \QQ\cdot\ell$$ its transcendental lattice, i.e.\
the minimal  sub-Hodge structure with non-trivial $(2,0)$-part. For example, since the original Hodge structure is assumed to be irreducible, we have $T_0=T$.
The Picard number of $\sigma_t$ is then defined  as  $$\rho(\sigma_t)\coloneqq\dim (T\oplus\QQ\cdot\ell)-\dim T_t.$$ 
We are used to think of the Picard number as the dimension of the subspace of all classes in $T\oplus\QQ\cdot\ell$ that are  $(1,1)$ with respect to $\sigma_t$. However,
due to the fact that the quadratic form is possibly non-degenerate,
this number can be bigger, cf.\ the discussion for $d=0$ in \S \ref{sec:HodgeBrauer} below.

With this definition, $\rho(\sigma_0)=1$ and 
$\rho(\sigma_t)=0$ for the very general $\sigma_t$. In fact, $\rho(\sigma_t)=1$ is the maximal value that can be attained due to the assumption on $\sigma_t$ to be brilliant.
Slightly  stronger, we have $\dim (T_t^\perp)\leq 1$ for all $t$. To see this adapt the proof of
\cite[Prop.\ 3.2]{HuyCM} to cover also the case $d<0$ and $d=0$.

The set of brilliant
deformations $\sigma_t$ of $\sigma_0$ with $\rho(\sigma_t)=1$
is countable and dense. Of particular interest to us are those for which
the transcendental lattice $T_t\subset T\oplus \QQ\cdot\ell$ has the same signature $(2,r-2)$ as $T_0=T$. Then clearly $\rho(\sigma_t)=1$ and the converse holds for $d\leq 0$.
We define the Noether--Lefschetz
locus as the following countable  dense subset of brilliant deformations of $\sigma_0$:
$${\rm NL}_\ell\coloneqq\{\,\sigma_t\mid {\rm sign}(T_t)={\rm sign}(T_0)=(2,r-2) \,\}.$$

\begin{lem}\label{lem:orthprojHodge}
For $\sigma_t\in {\rm NL}_\ell$ orthogonal projection $T_t\,\hookrightarrow T\oplus\QQ\cdot\ell\twoheadrightarrow T_0=T$ defines a bijection
$$T_t\congpf T$$
of $\QQ$-vector spaces, which for $d=0$ is an isometry of Hodge structures
 up to conjugation.
\end{lem}

\begin{proof} First observe that for $\rho(\sigma_t)=1$ the obvious
inclusion $P_{\sigma_t}\subset T_t\cap (P_{\sigma_0}\oplus \RR\cdot \ell)$  is an equality.
Next, since $T_t$ and $T_0$ are $\QQ$-vector spaces of the same dimension, it suffices to prove the injectivity of $T_t\to T$, which follows from $\ell\not\in T_t$ or, equivalently,
$\ell\not\in P_{\sigma_t}=T_t\cap (P_{\sigma_0}\oplus \RR\cdot\ell)$. 
Indeed, as $P_{\sigma_t}$ is positive definite, we clearly have
$\ell\not\in P_{\sigma_t}$ for $d\leq0$. For  $d>0$ it is implied by the 
assumption that $\sigma_t$ is  brilliant deformations. 

 For the second assertion in the case $d=0$ use that $\sigma_t=\sigma_0+ t\ell\in L_\ell$ or $\sigma_t=\bar\sigma_0+t\ell\in \bar L_\ell$ for some $t\in \CC$. Hence, the projection $T_t\to T$ maps $\sigma_t$ to $\sigma_0$ or $\bar\sigma_0$. 
 As for $d=0$ the set of $(1,1)$-classes is the same for all $\sigma_0+t\ell$, see
 \S \ref{sec:HodgeBrauer} below, also $(1,1)$-classes are preserved under the projection.
 \end{proof}

\subsection{}\label{sec:HodgeBrauer}
Let us now study the case of an isotropic class $\ell$, i.e.\ $d=0$. 
We will restrict to the connected component $L_\ell$, i.e.\ we only study brilliant deformations of the form $\sigma_t=\sigma_0+t\ell$ for some $t\in \CC$. Those parameterized by $\bar L_\ell$ are obtained by complex conjugation.

The first thing to observe is that, unlike the case $d\ne0$, the class $\ell$ is
always of type $(1,1)$, for $(\sigma_t.\ell)=0$. However, 
for the very general $\sigma_t\in L_\ell$, one has $T_t=T\oplus\QQ\cdot \ell$, cf.\ Lemma \ref{lem:NLBrauer}, and, therefore, in spite of being of type $(1,1)$ for $\sigma_t$, the class $\ell$ is contained in the transcendental lattice $T_t$. More generally,
the space of $(1,1)$-classes is the same for all $\sigma_t\in L_f$.

Next, any $\sigma_t\in L_\ell$  can be written
as $$\sigma_t=\sigma_0+(\sigma_0.B)\ell$$ for some $B\in T\otimes\CC$. Here,
the class $B$ is not uniquely defined, only its $(0,2)$-part $B^{0,2}\in T^{0,2}$ is.
This yields an identification
\begin{equation}\label{eqn:LlT02}
L_\ell\congpf T^{0,2},
\end{equation}
which is independent of the choice of $\sigma_0$.
The following description of the image of the Noether--Lefschetz locus  under 
(\ref{eqn:LlT02}) views the irreducible Hodge structure $T$ as a $\QQ$-linear
subspace of $T^{0,2}$ via orthogonal projection.

\begin{lem}\label{lem:NLBrauer}
A point $t\in L_\ell$ is contained in the Noether--Lefschetz locus ${\rm NL}_\ell$
if and only if there exists a class $B\in T$ (in the $\QQ$-vector space!) with $\sigma_t=\sigma_0+(\sigma_0.B)\ell$. In other words, under the identification
(\ref{eqn:LlT02}),
$${\rm NL}_\ell\cong T.$$
\end{lem}

\begin{proof} For any $B\in T$, the map $f_B\colon T\to T\oplus\QQ\cdot \ell$, $\alpha\mapsto \alpha+(\alpha.B)\ell$ defines a $\QQ$-linear isometric embedding
with $\sigma_t=\sigma_0+(\sigma_0.B)\ell\in f_B(T)\otimes \CC$ and, moreover, its image $f_B(T)\subset T\oplus \QQ\cdot\ell$ is a sub-Hodge structure with respect to $\sigma_t$.
Thus, $T_t\subset f_B(T)$ by minimality of $T_t$, which combined with $\rho(\sigma_t)\leq 1$ shows that
$T_t=f_B(T)$ is of signature $(2,r-2)$, i.e.\ $t\in {\rm NL}_\ell$.

Conversely, if $t$ is contained in the Noether--Lefschetz locus, then by virtue of Lemma
\ref{lem:orthprojHodge}, orthogonal projection yields an isomorphism $T_t\congpf T$. Its inverse is
a map $T\to T\oplus \QQ\cdot\ell$, $\alpha\mapsto \alpha+g(\alpha)\ell$ for some
linear map $g\colon T\to \QQ$. As $(~.~)$ is non-degenerate on $T$, 
there exists a (non-unique) $B\in T$ with $g(\alpha)=(\alpha.B)$ for all $\alpha$.
\end{proof}

If $T$ is associated with an integral Hodge structure, i.e.\ an integral lattice
$T_\ZZ\subset T$ with $T_\ZZ\otimes\QQ\cong T$ is fixed,
then its Brauer group can be defined and it comes in two flavors.
In the following definition we again use the injectivity of the orthogonal
projection $T\to T^{0,2}$ which allows us to view both, the $\ZZ$-module
$T_\ZZ$ and the bigger $\QQ$-vector space $T$, as 
subgroups of $T^{0,2}$.

\begin{definition}
 The \emph{analytic Brauer group} of an irreducible integral Hodge structure $T_\ZZ$ of K3 type as above is the quotient $${\rm Br}^{\rm an}(T_\ZZ)\coloneqq T^{0,2}/T_\ZZ.$$
 The \emph{algebraic Brauer group} ${\rm Br}(T_\ZZ)\subset {\rm Br}^{\rm an}(T_\ZZ)$
 is defined as the subgroup of all torsion elements or, equivalently, $${\rm Br}(T_\ZZ)\coloneqq T/T_\ZZ\subset T^{0,2}/T_\ZZ={\rm Br}^{\rm an}(T_\ZZ).$$
\end{definition}

Then Lemma \ref{lem:NLBrauer} expresses a relation between the
Noether--Lefschetz locus and the Brauer group:
$$\xymatrix@R=18pt{~~Ê~~{\rm NL}_\ell\ar@<-6ex>@{->>}[d]\ar@<5ex>@{->>}[d]~~\subset~ ~L_\ell\cong T^{0,2}\\
{\rm Br}(T_\ZZ)\subset{\rm Br}^{\rm an}(T_\ZZ)}$$

or, in other words,
$${\rm NL}_\ell/T_\ZZ\cong {\rm Br}(T_\ZZ),$$
where ${\rm NL}_\ell$ is considered as an affine space over $T_\ZZ$.

\subsection{}\label{sec:TwBrHs} There are various similarities between the three cases,
$d<0$, $d=0$, and $d>0$ and also a procedure to combine them.

\begin{remark}\label{rem:AO}
The most notable difference between the three cases is the signature of
$T\oplus\QQ\cdot\ell$. To simplify things, assume $T$ is two-dimensional. Then
$T\oplus\QQ\cdot\ell$ is positive definite for $d>0$, of signature $(2,1)$ for $d<0$,
and degenerate with one isotropic and two positive directions for $d=0$. This not only has consequences
for the description of the corresponding period domains as in Lemma \ref{lem:PeriodDomain}, but also for the possible existence of algebraic quotients.

To make this more transparent, assume 
that we start with an integral Hodge structure $T_\ZZ$, let
$T=T_\ZZ\otimes \QQ$, and assume $d=(\ell.\ell)\in \ZZ$. Then the
orthogonal group $\Gamma\coloneqq{\rm O}(T_\ZZ\oplus \ZZ\cdot \ell)$ of the
lattice $T_\ZZ\oplus\ZZ\cdot \ell$ is a finite group for $d>0$ and in many cases essentially trivial. On the other hand, for brilliant families of Dwork type, so $d<0$,
the group $\Gamma$ is infinite and $D_\ell/\Gamma$ will be the quotient of
 $\HH$ by a modular group. Often, as in the case of the classical Dwork family, a
 family given over $D_\ell$ will descend to an algebraic family of the quotient of
 $D_\ell$ by some subgroup of $\Gamma$ of finite index.
 
Something special happens in the case $d=0$. Here, $\Gamma$ is again an infinite group. Indeed, $\Gamma$ contains the $B$-field shifts $\alpha+n\ell\mapsto \alpha+((\alpha.B)+n)\ell$ for any $B\in T_\ZZ$. Note that this also works when $\dim T>2$
and that under the identification $D_\ell=L_\ell\cong T^{0,2}\cong\CC$ the action corresponds to translation by elements in $T_\ZZ\subset T^{0,2}$. However, unlike the case
$d<0$, the quotient $D_\ell/\Gamma$ usually does not have any reasonable
geometric structure.
\end{remark}

The three types of brilliant deformations can be combined as follows. 
We extend the Hodge structure on $T$ given by $\sigma_0$ to $$T_{\ell_1,\ell_2}\coloneqq T\oplus\QQ\cdot \ell_1\oplus\QQ\cdot \ell_2$$
by declaring both classes $\ell_1$ and $\ell_2$ to be of type $(1,1)$.
We assume $d\coloneqq (\ell_1.\ell_1)>0$, $(\ell_1.\ell_2)=0$, and $(\ell_2.\ell_2)=-d$. In particular, $f\coloneqq
\ell_1+\ell_2$ is an isotropic class. Analogously to $D_\ell\subset Q_\ell$, one defines $$D_{\ell_1,\ell_2}\subset Q_{\ell_1,\ell_2}\subset\PP(\CC\cdot\sigma_0\oplus\CC\cdot\bar\sigma_0\oplus\CC\cdot\ell_1\oplus\CC\cdot\ell_2)$$ by the conditions $(\sigma.\sigma)=0$
and $(\sigma.\bar\sigma)>0$. Clearly, the three types of period domains
$D_\ell$ with $(\ell.\ell)>0$, $=0$, and $<0$ considered
previously reappear as
curves $$D_{\ell_1},D_{\ell_2}, D_f\subset D_{\ell_1,\ell_2}.$$
More generally, for any rational (or real) linear combination $\ell=c_1\ell_1+c_2\ell_2$
the associated period domain $D_\ell$ describes a curve in $D_{\ell_1,\ell_2}$.
Note that if an arbitrary $\sigma\in D_{\ell_1,\ell_2}$ is written as $\sigma=a\sigma+b\bar\sigma+(c_1\ell_1+c_2\ell_2)$, then $\sigma$ can be viewed as a point in the one-dimensional
family of deformations $D_\ell$ with $\ell={c_1\ell_1+c_2\ell_2}$ as in \S \ref{sec:brilliant}, but typically $\ell=c_1\ell_1+c_2\ell_2$ is not even real.
In this context one would again say that $\sigma$ is a brilliant deformation of
$\sigma_0$ whenever $P_\sigma$ intersects $\RR\cdot\ell_1\oplusÊ\RR\cdot\ell_2$ trivially. In Remark \ref{rem:equatorflow} we have noted already that the equators
$S^1_{\ell_s}\subset D_{\ell_s}$ for $\ell_s=\ell_1+s\ell_2$ with $s\in [0,1)$ specialize
to the point $[f]\in Q_f$ for $s$ approaching $1$.

Putting the brilliant deformations associated with varying classes $\ell=c_1\ell_1+c_2\ell_2$
into one family shows the discrete Noether--Lefschetz loci ${\rm NL}_\ell$
as part of another continuous family. To make this precise, consider
$\ell'\in T\oplus\QQ\cdot\ell_1\oplus\QQ\cdot\ell_2$ not contained
in $\QQ\cdot\ell_1\oplus\QQ\cdot\ell_2$ such that
\begin{equation}\label{eqn:condlprime}
(\ell'.\ell')>0~\text{ and } ~(\ell'.f)>0.
\end{equation}
Then $D_{\ell_1,\ell_2}\cap\ell'^\perp\subset D_{\ell_1,\ell_2}$
is a one-dimensional family of deformations $\sigma$ of $\sigma_0$ 
for which $\ell'$ is of type $(1,1)$.

\begin{prop}\label{prop:NL}
For any $\ell'$ satisfying (\ref{eqn:condlprime}) and any rational
$\ell=c_1\ell_1+c_1\ell_2$, the curve $D_{\ell_1,\ell_2}\cap\ell'^\perp$ intersects the
family of brilliant deformations parameterized by
$D_\ell$ in the Noether--Lefschetz locus.
Conversely, if $\sigma_t\ne\sigma_0$ is a brilliant deformation contained in ${\rm NL}_\ell\subset D_\ell$, then $\sigma_t\in \ell'^\perp$ for an appropriate $\ell'$ as above.
\end{prop}

\begin{proof} Let $\sigma_t\in D_\ell$ be a brilliant deformation of $\sigma_0$ orthogonal
to $\ell'$. If $(\ell.\ell)>0$, then $\sigma_t$  is contained in the orthogonal complement of the plane $\langle\ell',\ell''\rangle_\QQ$, where $\ell''$ is a generator of  $\ell^\perp\subset\QQ\cdot\ell_1\oplus\QQ\cdot\ell_2$. As ${\rm sign}\langle\ell',\ell''\rangle_\QQ=(1,1)$ and $\sigma_t$ is brilliant, $T_t$ is of signature
$(2,r-2)$ and, therefore, $\sigma\in {\rm NL}_\ell$.
For $(\ell.\ell)=0$, i.e.\ $\ell$ a multiple of $f=\ell_1+\ell_2$ or 
$\ell_1-\ell_2$, we apply Lemma \ref{lem:NLBrauer}. For example, if
$\sigma_t=\sigma_0+t f$, then it suffices to show that $t=(\sigma_0.B)$
for some $B\in T$.
However, $(\sigma_t.\ell')=0$ and $(\ell'.f)\ne0$ allows us to write $t=-(\sigma_0.\ell')/(\ell'.f)$. 
Then let $B\coloneqq -(1/(\ell'.f))\bar\ell'$, where $\bar\ell'$ is the image of
$\ell'$ under the projection onto $T$.

Eventually, for $(\ell.\ell)<0$ the transcendental lattice $T_t$ is contained in
$T\oplus \QQ\cdot\ell$, which is of signature $(2,r-1)$. Since 
by assumption $\ell'$ is 
 not contained in $\QQ\cdot\ell_1\oplus\QQ\cdot\ell_2$, the transcendental lattice $T_t$ is of dimension $r$
and using $P_{\sigma_t}\subset T_t\otimes\RR$ its signature must be $(2,r-2)$.
Hence, also in this case, $\sigma_t\in {\rm NL}_\ell$.

To prove the converse, recall that $\sigma_t$ is contained in the Noether--Lefschetz locus if its transcendental lattice $T_t\subset T\oplus\QQ\cdot\ell$ has signature $(2,r-2)$. For $(\ell.\ell)>0$, its orthogonal complement $T_t^\perp$ is spanned by some $\ell'\in T\oplus\QQ\cdot\ell$ with $(\ell'.\ell')>0$, which , after changing its sign if necessary, also satisfies
$(\ell'.f)>0$. For $(\ell.\ell)<0$, one changes a generator of the orthogonal
complement $T_t^\perp\subset T\oplus\QQ\cdot\ell$ by a large multiple of a
class in $\ell^\perp\subset\QQ\cdot\ell_1\oplus\QQ\cdot\ell_2$ to ensure $(\ell'.\ell')>0$.
The condition $(\ell'.f)>0$ can then be achieved by a sign change if necessary.
If $(\ell.\ell)=0$, then $\ell$ is a multiple of $\ell_1\pm\ell_2$. We shall assume
$\ell=f=\ell_1+\ell_2$, the case $\ell=\ell_1-\ell_2$ is similar. Then the transcendental lattice is the image of $f_B$ for some $B\in T$, see 
Lemma \ref{lem:NLBrauer}. In other words, $T_t$ is orthogonal to any element of the form $\ell'\coloneqq -B+(1/d)\ell_1+kf$ with $(\ell'.f)>0$ and for $k\gg0$ one has $(\ell'.\ell')>0$. In all three cases, the assumption $\sigma_t\ne\sigma_0$ is needed to ensure that
$\ell'$ is not contained in $\QQ\cdot\ell_1\oplus\QQ\cdot\ell_2$.
\end{proof}

Rephrasing the above discussion yields the Hodge theory underlying Theorem \ref{thm:comp}. 

\begin{cor}\label{cor:HT04}
Consider a brilliant twistor deformation $\sigma_t\in D_{\ell_1}$ of $\sigma_0$ that is contained in ${\rm NL}_{\ell_1}$. Then there exists a class $\ell'\in T\oplus\QQ\cdot\ell_1\oplus\QQ\cdot\ell_2$ satisfying (\ref{eqn:condlprime})
and such that the curve $D_{\ell_1,\ell_2}\cap\ell'^\perp$ intersects
the Brauer family $L_f\subset D_f$ in a point in the Noether--Lefschetz locus, i.e.\
in a point of the form $\sigma_0+(\sigma_0.B)f$ for some $B\in T$.
\end{cor}

\begin{proof}
By virtue of the proposition, $\sigma_t\in \ell'^\perp$ for some $\ell'$ satisfying (\ref{eqn:condlprime}). Then $\sigma\coloneqq \sigma_0+(\sigma_0.B)f$, where
$B$ is the projection of $-(1/(f.\ell'))\ell'$ in $T$, is a brilliant deformation of $\sigma_0$ in the Brauer family. Moreover, $\sigma$ is contained in the Noether--Lefschetz locus 
${\rm NL}_f$ and, by construction, it is orthogonal to $\ell'$, which proves the assertion.
\end{proof}

\subsection{} 
As before, we consider an irreducible rational Hodge structure of K3 type $T$
and fix a generator $\sigma_0$ of $T^{2,0}$. 
Then the ring $$K(\sigma_0)\coloneqq{\rm End}_{\rm Hdg}(T)$$
of $\QQ$-linear endomorphisms $\varphi\colon T\to T$ of the Hodge structure 
$T$ is a field. The map $\varphi\mapsto \lambda$, with $\lambda$
determined by $\varphi(\sigma_0)=\lambda\sigma_0$, describes an embedding
$K(\sigma_0)\,\hookrightarrow\CC$ and according to \cite{Zarhin}, $K(\sigma_0)$ is either a RM or a CM field, i.e.\ $K(\sigma_0)$ is either a totally real field or a purely imaginary quadratic
extension $K^0(\sigma_0)\subset K(\sigma_0)$ of a totally real field $K^0(\sigma_0)$.

The Hodge structure $T$ is said to have CM if $K(\sigma_0)$ is a CM field and, additionally, $T$ is of dimension one as a vector space over $K(\sigma_0)$.

\begin{remark}
For a brilliant Hodge structure $\sigma_t$ contained in the Noether--Lefschetz locus, projection
yields a $\QQ$-linear isomorphism $\varpi_t\colon T_t\congpf T$, see Lemma \ref{lem:orthprojHodge}. In particular, any Hodge endomorphism
$\varphi\in K(\sigma_0)$ naturally describes a $\QQ$-linear endomorphism $\varphi_t$
of $T_t$. Observe that after linear extension, the projection $\varpi_t$ maps
$\sigma_t=a\sigma_0+b\bar\sigma_0+c\ell$ to $a\sigma_0+b\bar\sigma_0$.
Thus,  if $\varphi(\sigma_0)=\lambda\sigma_0$, then
$\varphi(\varpi_t(\sigma_t))=\varphi(a\sigma_0+b\bar\sigma_0)=a\lambda\sigma_0+b\bar\lambda\bar\sigma_0$ and, therefore, $\varphi_t(\sigma_t)=\lambda\sigma_t$ for any $\varphi$ with $\lambda\in \RR$. Hence, in this case the $\QQ$-linear map
$\varphi_t\colon T_t\to T_t$ respects the $(2,0)$-part. However, 
for $d\ne 0$ the map $\varphi_t$ is not an endomorphism of the Hodge structure $T_t$, as it does not respect the space of $(1,1)$-forms. For $d=0$, $\varphi_t$ is a Hodge endomorphism, which also follows from Lemma \ref{lem:orthprojHodge}.
\end{remark}

In spite of there not being an obvious way of associating to a Hodge endomorphism
$\varphi$ of $T$ a Hodge endomorphism $\varphi_t$ of $T_t$, this is possible whenever
$T$ has CM and $\varphi$ is contained in $K^0(\sigma_0)$. The proof of
\cite[Prop.\ 3.8]{HuyCM}, dealing with the case $d=(\ell.\ell)>0$,
carries over literally to the case $d< 0$ and yields the following.

\begin{prop}\label{prop:CMBrilliant}
 Assume $T$ is an irreducible Hodge structure of K3 type with complex multiplication. Then, for any brilliant deformation $\sigma_t$ of $\sigma_0$ as a Hodge structure on $T\oplus\QQ\cdot \ell$ that is contained in the Noether--Lefschetz
locus ${\rm NL}_\ell$, there exists a natural embedding
$$K^0(\sigma_0)\,\hookrightarrow K(\sigma_t)={\rm End}_{\rm Hdg}(T_t).$$
\end{prop}

\begin{proof}
Only the case $d<0$ needs to be checked. For a brilliant deformation $\sigma_t$
in ${\rm NL}_\ell$
there exists a class $\ell'\in T\oplus \QQ\cdot\ell$ with $m\coloneqq (\ell.\ell')\ne0$
and $T_t=\ell'^\perp$, which  is of signature $(2,r-2)$. Note that
unlike the case considered in \cite[\S 3]{HuyCM}, we have $(\ell'.\ell')<0$, but this
sign difference is of no consequence in the argument there, only
$m\ne0$ matters.
\end{proof}

The result immediately implies Theorem \ref{thm:mainRM}, see \S \ref{sec:K3Transl} for the translation.

\begin{cor}\label{cor:prop:CMBrilliant}
Under the assumption of the previous proposition, if $\sigma_0$ has complex multiplication, then also any brilliant deformation  $\sigma_t$ of $\sigma_0$ contained
in the Noether--Lefschetz locus has complex multiplication and the maximal
totally real subfields of $\sigma_0$ and $\sigma_t$ coincide.
\end{cor}

\begin{proof}
For dimension reasons, $T_t$ as a vector space over $K^0(\sigma_0)$
is of dimension two. Hence, according to \cite[Lem.\ 3.2]{vGeemen}, $T_t$ has complex multiplication and $K^0(\sigma_0)$ is its maximally totally real subfield.
\end{proof}

\begin{remark}
Note that the  CM fields $K(\sigma_0)$ and $K(\sigma_t)$ will typically be different.
In \cite[Cor.\ 3.10]{HuyCM} an equation of degree two 
$X^2+\gamma X+\delta=0$ is given that describes the extension
$K^0(\sigma_0)\cong K^0(\sigma_t)\subset K(\sigma_t)$. Unfortunately,  there is a factor $2$ missing
in the computation of $\delta$, which really should be $\delta=m^2-(1/2)d(\sigma_0.\bar\sigma_0)^{-1}(\alpha^2+\alpha^{-2}-2)$.\footnote{Thanks
to F.\ Vigan\`o \cite{Vigano} for pointing this out to me.} The same computation then holds for
any brilliant deformation in the Noether--Lefschetz locus as long as $d\ne0$. Note that
for $d=0$ we have $K(\sigma_0)\cong K(\sigma_t)$ for all $t$ in the Noether--Lefschetz locus.
\end{remark}


\section{Geometric brilliant families} 
In this section we translate the Hodge theory of the previous section
into geometric statements. We also indicate how to think of the Hodge theoretic 
Brauer family for K3 surfaces that are not elliptic.

\subsection{}\label{sec:K3Transl} We shall first explain how to reduce the geometric situation considered in the introduction and in \S  \ref{sec:Geom} to the purely Hodge theoretic context in  \S \ref{sec:brilliantHodge}. The arguments are straightforward and we will be brief.

Assume $S$ is a complex projective K3 surface with transcendental lattice $T(S)\coloneqq \NS(S)^\perp\subset H^2(S,\ZZ)$. Then  the real Hodge structure $T_\RR\coloneqq T(S)\otimes \RR$ and the rational Hodge structure $T\coloneqq T(S)\otimes\QQ$ satisfy the assumptions in \S \ref{sec:brilliant} and \S \ref{sec:RatHS}.
In particular, $T_\QQ$ is an irreducible Hodge structure. The generator $\sigma_0\in T^{2,0}$ is nothing but a non-zero holomorphic $(2,0)$-form on $S$.
For the additional class $\ell$ in (\ref{eqn:TRplusell}) and (\ref{eqn:TQplusell}) one picks
a class $\ell\in H^{1,1}(S,\ZZ)$. 

The three cases $(\ell.\ell)>0$, $=0$, and $<0$ are geometrically realized by
the constructions in \S \ref{sec:Dwork}, \S \ref{sec:GeomBrauer}, and \S \ref{ref:Twist}.

The surjectivity of the period map, see \cite[Ch.\ 7]{HuyK3} for the statement and references, immediately yields the following.
\begin{prop}
Consider the period domain $D_\ell$ in (\ref{eqn:Dell}). Then there exists
a smooth family of K3 surfaces $\ks\to D_\ell$ such that the fibre over
$\sigma_0$ is $S$ and parallel transport identifies $H^{2,0}(\ks_t)$ with 
the Hodge structure on $T\oplus \QQ\cdot\ell$ corresponding to $\sigma_t$.\qed
\end{prop}

For an ample class $\ell$ the family can alternatively be viewed as the twistor family
in \S \ref{ref:Twist}. If $\ell$ is the class of the fibre of an elliptic fibration of $S$, 
the family can be constructed geometrically and without using the surjectivity of the period map, see \S \ref{sec:GeomBrauer}. However, for $(\ell.\ell)<0$ there typically is no alternative geometric construction. Moreover, in this case, unless $\rho(S)=20$, the family
$\ks\to D_\ell$ will not descend to an algebraic family despite all its fibres being
polarized by any ample class in $\ell^\perp\cap H^{1,1}(S,\ZZ)$, see Remark \ref{rem:AO} and the discussion at the end of \S \ref{sec:Dwork}.

\begin{remark}
For $(\ell.\ell)\leq 0$, the family $\ks\to D_\ell$ is a brilliant deformation of $S$, i.e.\
$H^{2,0}(\ks_t)\subset (H^{2,0}\oplus H^{0,2})(S)\oplus\CC\cdot\ell$ with $\ell\not\in
(H^{2,0}\oplus H^{0,2})(\ks_t)$. For $(\ell.\ell)>0$, the last condition forces us to restrict
to the union of the two hemispheres $D_\ell\setminus\ S^1_\ell$, which we shall henceforth do.

Note that by definition $t\in {\rm NL}_\ell$ if and only of $\ks_t$ is algebraic with $\rho(\ks_t)=\rho(S)$, which justifies to call ${\rm NL}_\ell$ the Noether--Lefschetz locus
of the brilliant family.
\end{remark}

\smallskip

\noindent
{\it Proof of Theorem \ref{thm:mainRM}}
Apply Proposition \ref{prop:CMBrilliant} and Corollary \ref{cor:prop:CMBrilliant}.
\qed
\smallskip

Let us now focus on the case $(\ell.\ell)=0$. As remarked already in \S 
\ref{sec:GeomBrauer}, the two Brauer groups ${\rm Br}(S)$ and ${\rm Br}(T(S))$
differ by a finite group. More precisely, there is a surjection ${\rm Br}^{\rm an}(T(S))\twoheadrightarrow {\rm Br}^{\rm an}(S)$ with a finite kernel and inducing a surjection
between the algebraic Brauer groups.  However, as we rather work with the brilliant family over $H^{2,0}(S)\cong T^{2,0}\cong \CC$, the difference
is of no importance to us. If $\ell$ is the class $f$ of an elliptic fibration $S\to \PP^1$, then the Brauer family $\ks\to \CC$, see (\ref{eqn:BrauerFam}), is the geometric realization of the Hodge theoretic Brauer family. This yields

\begin{cor}\label{cor:EllBrauer}
Let $S$ be a complex projective elliptic K3 surface. Then the
Brauer family $\ks\to \CC\cong H^{2,0}(S)$ is a brilliant deformation of $S$
such that its fibre $\ks_t$ is isomorphic to the K3 surface $S_\alpha$, where
$\alpha$ is the image of $t$ under $H^{2,0}(S)\twoheadrightarrow {\rm Br}^{\rm an}(S)\cong\Sha^{\rm an}(S)$. Furthermore,
$\ks_t$ is algebraic if and only if $t\in T(S)\otimes\QQ$ or, equivalently, if $\alpha\in {\rm Br}(S)$.\qed
\end{cor}

In other words, the fibre $\ks_t$ is algebraic if and only if after parallel transport $H^{2,0}(\ks_t)$ is generated by $\sigma_0+(\sigma_0.B)f$ for some rational cohomology class $B\in H^2(S,\QQ)$.

\smallskip
\noindent
{\it Proof of Theorem \ref{thm:NLvsBr}}
In order to apply the Hodge theory in \S \ref{sec:TwBrHs}, we pick an ample line bundle $L$ on
$S$ and let $\ell_1\coloneqq m\cdot {\rm c}_1(L)$ with $m\coloneqq (f.{\rm c}_1(L))/ ({\rm c}_1(L).{\rm c}_1(L))$. Then $\ell_2\coloneqq f-\ell_1$ satisfies the assumption $(\ell_2.\ell_2)=-(\ell_1.\ell_1)$.

Then the first assertion in Theorem  \ref{thm:NLvsBr} is the observation  in Remark \ref{rem:equatorflow}, while the second follows from Proposition \ref{prop:NL}
applied to $\ell=\ell_1$.
\qed
\smallskip

\smallskip
\noindent
{\it Proof of Theorem \ref{thm:comp}}
The result is a direct consequence of Corollary \ref{cor:HT04}. The class $\ell'$ there satisfies $(\ell'.\ell')>0$. Thus, it is a $(1,1)$-class
in the positive cone of all fibres of $\ks\to D_{\ell_1,\ell_2}\cap\ell'^\perp$,
which therefore are all algebraic
\qed
\subsection{}\label{sec:higherDimMS}

In order to realize the Hodge theoretic description of the Brauer family in  \S \ref{sec:HodgeBrauer} geometrically, we need to give the additional isotropic
class $\ell$ a geometric meaning. In \S \ref{sec:GeomBrauer} and \S\ref{sec:K3Transl}
this was achieved for elliptic K3 surfaces $S\to\PP^1$ by letting $\ell$ be the fibre
class $f$. Of course, a general K3 surface is not elliptic, i.e.\
there is no isotropic $(1,1)$-class in $H^2(S,\ZZ)$. However, there is an easy way out by
passing to higher-dimensional hyperk\"ahler manifolds endowed with a Lagrangian
fibration.
To this end, we first note the following standard fact.

\begin{prop}
Let $S$ be a projective K3 surface. Then there exists a projective hyperk\"ahler manifold
$X$ together with a Lagrangian fibration $X\to \PP^n$ with a section and
an algebraic  Hodge isometry $T(S)\cong T(X)$.
\end{prop}

\begin{proof} Consider  a primitive ample line bundle $L$ on $S$
and a Mukai vector $v\in H^*(S,\ZZ)$ of the form $v=(0,{\rm c}_1(L),\chi)$.
Then the moduli space  $X$ of sheaves $E$ with
$v(E)=v$ that are Gieseker stable with respect to a $v$-generic polarization on $S$
is a projective hyperk\"ahler manifold, see \cite[Prop.\ 10.2.5]{HuyK3} for the statement and references.
Furthermore, there exists a Hodge isometry $H^2(X,\ZZ)\cong v^\perp\subset \widetilde H(S,\ZZ)$, where the left hand side is endowed with the Beauville--Bogomolov form and the right
hand side is the orthogonal complement of $v$ in the Mukai lattice, see \cite{OG}.
Moreover, the isomorphism is algebraic, i.e.\ it is induced by an algebraic class on the product $S\times X$, and, thus, yields an algebraic Hodge isometry between the transcendental lattices $T(S)\cong T(X)$. Finally, mapping a sheaf to its Fitting support 
yields a Lagrangian fibration $X\to \PP^n\cong |L|$, which for $\chi=(L.L)/2$ comes with a section given by the sheaves $L|_C$ for $C\in |L|$. 

Alternatively, one could work with any primitive $v$ with a square zero $(1,1)$-class in $v^\perp$ and use \cite[Thm.\ 1.5]{BayerMacri}
to produce a birational Lagrangian fibration.
%
\end{proof}

Clearly, there is no canonical choice for $X$, in fact not even for its dimension. But note that this happens already for the case of elliptic K3 surfaces considered before, as frequently K3 surfaces admit more than one elliptic fibration.

The analogue of Corollary \ref{cor:EllBrauer} for elliptic K3 surfaces is then the following.

\begin{cor}
Let $S$ be a projective K3 surface with $(H^{2,0}\oplus H^{0,2})(S)\cap H^2(S,\QQ)=0$.
  
{\rm (i)} There exists a compact Lagrangian
hyperk\"ahler manifold $\pi\colon X\to \PP^n$ together with an algebraic Hodge isometry
$T(S)\cong T(X)$ and a brilliant deformation  $\kx\to \CC$ of $X$ of Brauer type that is
endowed with a family of Lagrangian fibrations $\kx_t\to\PP^n$.

{\rm (ii)} If $\alpha_t\in {\rm Br}(X)$ is the image of a point $t\in \CC$ under the natural projection $H^{0,2}(S)\cong\CC\twoheadrightarrow {\rm Br}^{\rm an}(S)\cong {\rm Br}^{\rm an}(X)$,
then the Lagrangian fibration $\pi_t\colon\kx_t\to\PP^n$ corresponds to the
Tate--{\v{S}}afarevi{\v{c}} twist of $X\to \PP^n$ associated to $\alpha_t$.

{\rm (iii)}
The hyperk\"ahler manifold $\kx_t$ is projective if and only if $t$ is contained in the Noether--Lefschetz locus or, equivalently, if its image $\alpha_t$ is contained in the algebraic
Brauer group ${\rm Br}(X)\subset {\rm Br}^{\rm an}(X)$.
\end{cor}

To spell this out a bit more, observe that $T(S)\cong T(X)$ yields a natural isomorphism 
(up to a finite kernel) between
the algebraic and  the analytic Brauer groups of $S$ and $X$. The statement that $\kx\to\CC$ is a brilliant deformation
of Brauer type means that there exists an isotropic class $f\in H^{1,1}(X,\ZZ)$,
namely  $\pi^\ast{\rm c}_1(\ko(1))$, 
such that $H^{2,0}(\kx_t)\subset H^{2,0}(X)\oplus H^{0,2}(X)\oplus\CC\cdot f$.
This allows one to use the identification $L_f\cong H^{0,2}(X)\cong H^{0,2}(S)$
in (\ref{eqn:LlT02}) and the surjection ${\rm NL}_f\twoheadrightarrow {\rm Br}(X)\cong {\rm Br}(S)$.

\begin{proof} The idea of the proof is similar to the one for elliptic K3 surfaces.
The existence of $X\to \PP^n$ is guaranteed by the previous proposition.
For the construction of the family $\kx\to \CC$ we refer to \cite[\S\S7.1-7.2]{Mark},
it is analogous to the one for elliptic K3 surfaces \cite[Ch.\ 1.5]{FM}. To prove (iii), one needs to relate projectivity of $\kx_t$ to the existence of $(1,1)$-classes of
positive Beauville--Bogomolov square, for which we use the criterion
\cite[Thm.\ 3.11]{HuyInv}.\end{proof}

 Ideally one would like to phrase (ii) in terms of an isomorphism ${\rm Br}^{\rm an}(X)\cong\Sha^{\rm an}(X)$ (and similar for the algebraic variants),
but the situation is more complicated in higher dimensions.
Also note that the assumption on $(H^{2,0}\oplus H^{0,2})(S)$ holds for most K3 surfaces and is probably not needed. It is used to ensure that all fibres in the Brauer family obtained by twisting
the Lagrangian fibration $X\to \PP^n$ are actually hyperk\"ahler.

\begin{remark}
The main result of \cite{Mark} roughly says that any Lagrangian fibration of a hyperk\"ahler manifold that is deformation equivalent to a Hilbert scheme of a K3 surface
is obtained as a fibre $\kx_t$ of a  family of the above type.
\end{remark}

\begin{remark}
The situation in positive characteristic shows certain features that are definitely not
mirrored by Hodge theory. For example, the construction of
`twistor spaces'  in \cite{BL} starts with a supersingular K3 surface $S$ together
with a fixed isotropic Mukai vector $v$ and interpretes the collection of moduli spaces
of stable sheaves  that are twisted with respect to the varying Brauer class $\alpha$
as a family of untwisted K3 surfaces. In characteristic zero, this does not work unless the Mukai vector $v=(r,\ell,s)$ satisfies $r=0$. Indeed, if an $\alpha$-twisted sheaf exists at all, then the order of $\alpha$ must divide $r$, which is impossible to ensure for varying Brauer classes.
\end{remark}


\end{document}